\documentclass[12pt]{article}
\usepackage{amsmath,amssymb,amsthm,latexsym,euscript,amscd}
\usepackage[english]{babel}
\usepackage[all]{xy}

\newtheorem{Theorem}{Theorem}
\newtheorem{Proposition}[Theorem]{Proposition}

\newtheorem{Lemma}[Theorem]{Lemma}

\newtheorem{Conjecture}[Theorem]{Conjecture}

\newcommand{\g}{{\mathfrak g}}
\newcommand{\gl}{{\mathfrak gl}}
\newcommand{\ttt}{{\mathfrak t}}

\begin{document}

\title{Symplectic geometry of unbiasedness and critical points of a potential}
\author{Alexey Bondal\thanks{ Steklov Institute of Mathematics, Moscow, and Kavli Institute for the Physics and Mathematics of the Universe (WPI), The University of Tokyo, Kashiwa, Chiba 277-8583, Japan, and National Research University Higher School of Economics, Russian
Federation, and The Institute for Fundamental Science, Moscow}~ and Ilya Zhdanovskiy\thanks{MIPT, Moscow, and HSE Laboratory of algebraic geometry, Moscow}}

\date{}

\maketitle

\begin{abstract}

The goal of these notes is to show that the classification problem of algebraically unbiased system of projectors has an interpretation in symplectic geometry. This leads us to a description of the moduli space of algebraically unbiased bases as critical points of a potential functions, which is a Laurent polynomial in suitable coordinates. The Newton polytope of the Laurent polynomial is the classical Birkhoff polytope, the set of double stochastic matrices. Mirror symmetry interprets the polynomial as a Landau-Ginzburg potential for corresponding Fano variety and relates the symplectic geometry of the variety with systems of unbiased projectors.
\end{abstract}

\section{Introduction}
The classification problem of systems of unbiased bases attracted a lot of attention in physics literature.
In its algebraic version, this is about complete systems of orthogonal rank 1 projectors $\{p_i\}$ and $\{q_j\}$ in a finite dimensional vector space of dimension $n$ satisfying the condition ${\rm Tr}p_iq_j=\frac 1n$. An equivalent formulation of the problem is to classify pair of Cartan subalgebras in the Lie algebra $sl(n)$ which are orthogonal with resect to Killing form. It is also important to classify bigger sets of systems of projectors which are pairwise algebraically unbiased.

In \cite{BZh} we introduced algebras $B_{\Lambda}(\Gamma )$, where $\Gamma$ is a simply laced graph with weights of edges $\Lambda$, and gave an interpretation of this problem in the context of representation theory of these algebras. For the case of two systems of projectors the graph is the full bipartite graph with $n$ vertices in both rows, all weights on edges are $\frac 1n$, and representations under consideration are of dimension $n$. More general weights appear when we put condition ${\rm Tr}p_iq_j=\lambda_{ij}$. Algebras happened to be homotopes of Poincare groupoids of graphs associated to discrete Laplace operators. In that respect the theory is closely related to discrete harmonic analysis on graphs and to the gluing of t-structures in the style of \cite{BBD}.

In this paper, we develop the symplectic approach to the problem. The symplectic structure on the variety of projectors of rank 1, considered as a coadjoint orbit, is relevant to the problem. If we assume orthogonal projectors $q_j$ to be fixed and vary the projectors $p_i$ without requiring them to be orthogonal, then the equations on unbiased systems distinguish a Lagrangian fibre of an integrable system on the product $Y$ of coadjoint orbits corresponding to projectors. The locus of systems of orthogonal projectors defines another Lagrangian subvariety $X$ in $Y$. The problem is then to find the locus of intersection of these two Lagrangian subvarieties. We construct a symplectic open embedding of the cotangent bundle to $X$ into $Y$.

The Lagrangian fiber is then interpreted as a (multi-valued) section of the cotangent bundle to $X$ given as a differential of a potential function on $X$.
If the transition matrix from projectors $\{q_i\}$ to $\{p_i\}$ is $g$, then the exponent of the function is:
$$
E=\frac{{\rm det}g}{\prod (g_{ij})^{\frac 1n}}
$$
It is a Laurent polynomial in a suitable system of coordinates. Critical points of the function are solutions to our problem.

We consider also some generalizations to the case of non-complete system of projectors $\{p_i\}$ $i=1,\dots , k$ and complete systems of projectors $\{q_j\}$, $j=1,\dots , n$. The corresponding graph $\Gamma$ is the full bipartite graph $\Gamma_{kn}$ with $k$ vertices in one row and $n$ vertices in the other. Also we allow arbitrary constants $\lambda_{ij}$ instead of $\frac 1n$ in the above equation for traces. For our analysis, we introduce another algebra $A_{n,{\bar \Lambda}}$ together with a homomorphism $\psi: A_{n,{\bar \Lambda }}\to B_{\bf \Lambda}(\Gamma_{kn} )$, which corresponds to summation of projectors $p_i$ to a projector of rank $k$. We prove that the natural symplectic form that exists on the moduli of representations of $A_{n,{\bar \Lambda}}$ goes to zero when restricted to the moduli of $B_{\bf \Lambda}(\Gamma_{kn})$ representaions. We conjecture that the image of every component of one moduli is Lagrangian in the other moduli.

Critical points of Laurent polynomials appear in the study of mirror symmetry for toric varieties and their deformations. The cones of the fan of the toric variety are spanned by faces of the Newton polytope of the Laurent polynomial. In our case this polytope is the classical Birkhoff polytope. Its points are double stochastic matrices studied in probability theory. The corresponding toric variety is Gorenstein Fano variety with terminal singularities. According to mirror symmetry, the symplectic geometry of this variety is captured by critical points of the potential function under an appropriate choice of symplectic form and desingularization of the variety.

This connection deserve further study.

{\bf Acknowledgements.} 
We are grateful to Alexei Rosly for useful discussions. This work was partly done during authors visit to Kavli IPMU and was supported by World Premier International Research Center Initiative (WPI Initiative), MEXT, Japan.
The reported study was partially supported by RFBR, research projects 13-01-00234, 14-01-00416 and 15-51-50045. The article was prepared within the framework of a subsidy granted to the HSE by the Government of the Russian Federation for the implementation of the Global Competitiveness Program.

\section{Algebraically unbiased projectors and representation theory}
\label{sectprojectors}
Here we recall the notion of algebraically unbiased projectors and show how to interpret them in terms of representations of suitable algebras associated to graphs.

\subsection{Algebraically unbiased projectors and orthogonal Cartan subalgebras}
\label{aup}

Let $V$ be a $n$-dimensional space over field of characteristic zero.


Two minimal (i.e. rank 1) projectors $p$ and $q$ in $V$ are said to be {\it algebraically unbiased} if
\begin{equation}
\label{unb} {\rm Tr}(pq) = \frac{1}{n}
\end{equation}
Equivalently, this reads as one of the two equivalent algebraic relations:
\begin{equation}\label{aunbias1}
pqp = \frac{1}{n}p,
\end{equation}
\begin{equation}\label{aunbias2}
qpq = \frac{1}{n}q.
\end{equation}

We will also consider {\it
orthogonal} projectors. Orthogonality of $p$ and $q$ is algebraically expressed as
\begin{equation}\label{orthog}
pq = qp = 0
\end{equation}

Two complete (i.e. of cardinality $n$) systems of minimal orthogonal projectors $(p_1,...,p_n)$ and $(q_1,...,q_n)$  are said to be {\em algebraically unbiased} if
$p_i$ and $q_j$ are algebraically unbiased for all pairs $(i,j)$.


Classification of pairs of algebraically unbiased complete systems of projectors is a notoriously difficult problem. Another difficult problem is to find the maximal set of pairwise algebraically unbiased complete systems of projectors. Both problems have the following interpretation in the theory of Lie algebras.

Consider  a simple Lie algebra $L$ over an algebraically closed field of characteristic zero.
Let $K$ be the Killing form on $L$.
The following definitions were introduced by J.G.Thompson in 1960, in course of his study of integer quadratic lattices.

{\bf Definition.} Two Cartan subalgebras $H_1$ and $H_2$
in $L$ are said to be {\it orthogonal} if $K(h_1,h_2) = 0$ for all $h_1 \in H_1,
h_2 \in H_2$.

{\bf Definition.} Decomposition of $L$ into the direct sum
of Cartan subalgebras $L = \oplus^{h+1}_{i=1}
H_i$ is said to be {\it orthogonal} if
$H_i$ is orthogonal to $H_j$, for all $i \ne j$.

Intensive study of orthogonal decompositions has been undertaken since then (see the book \cite{KT} and references therein).
For Lie algebra $sl(n)$, A.I. Kostrikin et all arrived to the following conjecture, called {\it Winnie-the-Pooh Conjecture} (cf. {\em ibid.} where, in particular, the name of the conjecture is explained by a wordplay in the
Milne's book in Russian translation).

\begin{Conjecture}
Lie algebra $sl(n)$ has an orthogonal decomposition if and only if $n = p^m$, for a prime number $p$.
\end{Conjecture}

The non-existence of an orthogonal decomposition for $sl(6)$, when $n=6$ is the first number which is not a prime power is still open. It is also important to find the maximal number of pairwise orthogonal Cartan subalgebras in $sl(n)$ for any given $n$ as well as to classify them up to obvious symmetries.

Recall the interpretation of the problem in terms of algebraically unbiased systems of rank 1 projectors. This was discovered by the first author about 25 years ago (cf. \cite{KT}).

Let $sl(V)$ be the Lie algebra of traceless operators in $V$. Killing form is given by the trace of product of operators.
A Cartan subalgebra $H$ in $V$ defines a unique complete system of minimal orthogonal projectors in $V$. Indeed, $H$ can be extended to the Cartan subalgebra $H'$ in $gl(V)$ spanned by $H$ and the identity operator $E$.
Rank 1 projectors in $H'$
are pairwise orthogonal and comprise the required system.

If $p$ is a minimal projector in $H'$, then trace of $p$ is 1, hence, $p-\frac 1nE$ is in $H$. If projectors $p$ and $q$ are associated to orthogonal Cartan subalgebras, then
$$
{\rm Tr}(p-\frac 1nE)(q-\frac 1nE)=0,
$$
which is equivalent to $p$ and $q$ to be algebraically unbiased.

Therefore, an orthogonal pair of Cartan subalgebras is in one-to-one correspondence with a pair of algebraically unbiased complete systems of orthogonal projectors. Similarly, orthogonal decompositions of $sl(n)$ correspond to a set of $n+1$ of pairwise algebraically unbiased complete systems of orthogonal projectors. In particular, it is clear from this correspondence that the number of pairwise algebraically unbiased complete systems never exceeds $n+1$.

\subsection{Graphs and algebra $B_{\Lambda}(\Gamma )$}

In the analysis of the problem, it makes sense to consider not only complete systems of orthogonal projectors, but also study mutual unbiasedness for other configurations of projectors. Thus, we come to the problem of studying the sets of projectors where every pair satisfies either conditions (\ref{aunbias1}-\ref{aunbias2}) or (\ref{orthog}). This leads us to assign an initial graph data to the problem and to the representation theory of suitable algebras related to this data.

Assume we are given a simply laced graph $\Gamma$ with a finite number of vertices and no loop (i.e. no edge with coinciding ends). Consider a finite dimensional vector space $V$. Assign a rank 1 projector in $V$ to every vertex of the graph. If two vertices are related by
an edge, then we require the corresponding two projectors to be algebraically unbiased. If there is no edge between the vertices, then we require the projectors to be orthogonal. The problem is to classify, for a given graph, all possible systems of projectors satisfying the required conditions modulo automorphisms of $V$.

It makes sense to consider the problem in a more general context.
Let us assign to every edge $(ij)$ of the graph a nonzero weight $\lambda_{ij}$ of the base field. We put $\lambda_{ij}=\lambda_{ji}$. The set of all weights is denoted by ${\bf \Lambda}$.

Define
algebra $B_{\bf \Lambda}(\Gamma )$ as a unital algebra over the base field with generators $x_i$, numbered by
vertices $i$ of $\Gamma$, which subject to relations:
\begin{itemize}
\item{$x^2_i = x_i$, for every vertex $i$,}
\item{$x_ix_jx_i = {\lambda}_{ij}x_i$, $x_jx_ix_j = {\lambda}_{ij}x_j$, if $i$ and $j$ are adjacent
in $\Gamma$,}
\item{$x_ix_j = x_jx_i = 0$, if there is no edge connecting $(ij)$ in $\Gamma$.}
\end{itemize}

Various properties of this algebra, as well as relation of it with Hecke algebras of graphs and Poincare groupoids of graphs, are described in \cite{BZh}.

We are interested in finite dimensional representations of $B_{\bf \Lambda}(\Gamma )$.
If graph is connected, then all generators $x_i$ are presented by projectors of the same rank. We will call this rank by the {\em rank of the representation}. The category of representation of $B_{\bf \Lambda}(\Gamma )$ can be understood (see \cite{BZh}) in terms of gluing of t-structures \cite{BBD}.

Consider graph $\Gamma$ the full bipartite graph $\Gamma_{n,n}$ with $n$ vertices in both rows. If we put all weights ${\lambda}_{ij}=\frac 1n$, then the classification problem for rank 1 $n$-dimensional representations of $B_{\bf \Lambda}(\Gamma_{nn} )$ is equivalent to the problem of classifying pairs of algebraically unbiased complete systems of orthogonal projectors. If $\Gamma$ is the graph with $n+1$ rows and $n$ vertices in every row, such that all vertices in different rows are connected by an edge and vertices in the same row have no connecting edge, and all weights are again ${\lambda}_{ij}=\frac 1n$, then $n$-dimensional representations of algebra $B_{\bf \Lambda}(\Gamma )$ correspond to the maximal system of pairwise orthogonal projectors, i.e. orthogonal decomposition of the Lie algebra $sl(n)$.

\subsection{Mutually unbiased bases and configurations of lines}

Now assume that the vector space is over ${\mathbb C}$ and it is endowed with a hermitian metric. Let all projectors be Hermitian, i.e. $p=p^\dag$.
A version of the problem on classification of pairwise unbiased systems of projectors 
attracted a lot of attention from physicists. 
If orthonormal bases compatible with an algebraically unbiased pair of complete system of orthogonal Hermitian projectors is chosen, then we got what is called in physics literature mutually unbiased bases.

Let $V$ be an $n$ dimensional complex space with a fixed
Hermitian metric $\langle \ ,\ \rangle $. Two orthonormal Hermitian
bases $\{e_i\}$ and $\{f_j\}$ in $V$ are said to be {\em mutually
unbiased} if, for all $(i,j)$,
\begin{equation}\label{ef}
|\langle e_i, f_j \rangle|^2 =\frac 1n
\end{equation}

They have numerous applications in the quantum information theory and related subjects (cf. \cite{BZh}).

The systems of orthogonal projectors related to mutually unbiased bases are algebraically unbiased systems of Hermitian projectors. The moduli spaces of algebraically unbiased systems of orthogonal projectors is a complexification of the moduli of mutually unbiased bases, in the sense that the latter moduli is an open subset in the stable locus of the anti-holomorphic involution on the former space, which is defined by action on the projectors by $p\mapsto p^\dag$.

In the same spirit, the moduli of rank 1 representations of the algebra $B_{\bf \Lambda}(\Gamma )$ can be interpreted as a complexified moduli of solutions for the following problem. Assign a complex line in a Hermitian space to every vertex of the graph $\Gamma$. Put the condition that two lines are Hermitian orthogonal if there is no edge linking the corresponding vertex of the graph. Put condition on  the angle $\alpha_{ij}$ between lines corresponding to vertices $i$ and $j$:
$$
{\rm cos}\ \alpha_{ij}=\lambda_{ij},
$$
where $\lambda_{ij}$ is the weight of the edge $(ij)$ in the graph. Clearly, we should assume that $\lambda_{ij}$ are real numbers in this case.

This problem for a relatively simple graph of one cycle of order 5 is the subject of intricate study in classical papers by J.Napier \cite{N} and C.F.Gauss \cite{G} (see also \cite{Sch}).

\section{Algebraically unbiased projectors and symplectic geometry}

Here we undertake the study of the moduli of representations for algebra $B_{\bf \Lambda}(\Gamma )$ for bipartite graphs $\Gamma$ in the framework of symplectic geometry.

\subsection{Bipartite graphs and an integrable system}
Fix positive integers $k$ and $n$ with $k\le n$.
Let $\Gamma =\Gamma_{kn}$ be the full bipartite graph with $k$ vertices in one row and $n$ vertices in the other. Here we discuss the moduli space, ${\cal M}_B$, of $B_{\bf \Lambda}(\Gamma_{kn} )$ representations of dimension $n$ and rank 1. By the moduli spaces of representations we will always mean the GIT quotient of the affine variety of representations in a space with given basis by the action of ${\bf GL}(V)$.

Denote by $p_i$ projectors from the first row, and by $q_j$ projectors from the second row. Our initial data include the matrix $\Lambda$ of coefficients $\lambda_{ij}$ the weights of the edge $(ij)$, where $i$ is the index of a vertex from the first row and $j$ from the second one. If we are interested in $n$-dimensional representations, then clearly we should put conditions:
\begin{equation}\label{sumi}
\sum_i\lambda_{ij}=1
\end{equation}
for every $j$. If $k=n$, then we should put also conditions:
\begin{equation}\label{sumj}
\sum_j\lambda_{ij}=1
\end{equation}
for every $i$.

We shall assume from now on that the base field is ${\mathbb C}$.
There is an integrable system related to the moduli space ${\cal M}_B$.
Consider a symplectic variety $Y$, the product of $k$ copies of the coadjoint orbit of projectors of rank 1, which we denote by ${\cal O}_{p_i}$:
$$
Y={\cal O}_{p_1}\times \dots \times {\cal O}_{p_k}.
$$

Condition $p_iq_jp_i=\lambda_{ij}p_i$ can be read as:
\begin{equation}\label{trpq}
{\rm Tr} p_iq_j=\lambda_{ij}.
\end{equation}
If we assume the projectors $q_j$ to be fixed and projectors $p_i$ to be varied in the $i$-th orbit ${\cal O}_{p_i}$, then functions ${\rm Tr} p_iq_j$ commute with respect to the Poisson structure on $Y$, in this sense, they define an integrable system. Fix a basis $\{e_i\}$ compatible with projectors $q_i$ and denote by $T=({\mathbb C}^*)^n$ the Cartan torus of diagonal matrices. The map into ${\mathbb C}^{kn}$ defined by ${\rm Tr} p_iq_j$ is actually the moment map for the action of $T^k$ on $Y$ by independent conjugations of projectors $p_i$. It takes values in the subspace ${\mathbb C}^{k(n-1)}$ in view of (\ref{sumi}), for $k<n$, and in the subspace ${\mathbb C}^{(n-1)^2}$, for $k=n$, in view of (\ref{sumj}). Thus, the Lagrangian fibers of the integrable system are algebraic tori, the orbits of the torus $T^k$.

Consider the product of $Y$ with the coadjoint orbit ${\cal O}_{-P}$ of negated projectors $-P$ (where $P^2=P$) of rank $k$. Let $X$ be the pre-image of {\bf 0} under the moment map
$$
\mu :Y\times {\cal O}_{-P}\to \g^*
$$
for the diagonal action of $G={\bf GL}(n)$ on this product, where $\g =\gl (n)$. Note that $\mu (p_1,\dots , p_k, -P)=\sum p_i-P$.

We leave the proof of the following lemma for the reader.

\begin{Lemma}
The sum of $k$ rank 1 projectors is a projector of rank $k$ if and only if the rank 1 projectors are pairwise orthogonal.
\end{Lemma}

This implies that $X$ is identified with the moduli space of systems of $k$ orthogonal projectors in a space with fixed basis.

\begin{Lemma}
$X$ is Lagrangian in $Y\times {\cal O}_{-P}$.
\end{Lemma}
\begin{proof}
Indeed, all systems of $k$ orthogonal projectors are in one orbit of ${\bf GL}(n)$, hence they are in the fibre of the projection map $\pi : Y\times {\cal O}_{-P}\to Y\times {\cal O}_{-P}/{\bf GL}(n)$. Since the arrows $\mu$ and $\pi$ comprise a Poisson pair of maps (i.e. pull-backs of algebras of functions along $\mu$ and $\pi$ are mutual annihilators), the symplectic form on $Y$ restricts by zero to the intersection of $\mu$-fiber with $\pi$-fiber. It implies that $X$ is isotropic and it is easy to see that it is of half dimension of $Y$.
\end{proof}
Let ${\rm pr}_1:Y\times {\cal O}_{-P}\to Y$ be the projection along the first coordinate.
Denote by $F_{\lambda_{ij}}\subset Y$ the fiber of the above integrable system over the point $(\lambda_{ij})$.
\begin{Proposition}\label{mb}
Moduli space ${\cal M}_B$ is identified with  ${\rm pr}_1^{-1}F_{\lambda_{ij}}\cap X$ modulo the action of one copy of $T$ that conjugates all $p_i$'s simultaneously.
\end{Proposition}
\subsection{Algebra $A_{n,{\bar \Lambda }}$ and Lagrangian image of ${\cal M}_B$}

We will also consider the algebra $A_{n,{\bar \Lambda }}$, where ${\bar \Lambda}=(\lambda_1,\dots , \lambda_n)$ is a sequence of elements in ${\mathbb C}^*$, with generators $q_i$, for $i=1,\dots , n$, and $P$, with relations:
$$
q_i^2=q_i,\ \ P^2=P,\ \ q_iq_j=0,\ {\rm for}\ i\ne j,\
q_iPq_i=\lambda_iq_i.
$$

An $A_{n,{\bar \Lambda }}$-representation is said to be {\em admissible} if it is of dimension $n$ and $q_i$'s are represented by projectors of rank 1 and $P$ by a projector of rank $k$. Let ${\cal M}_A$ be the moduli of admissible $A_{n,{\bar \Lambda }}$-representations. One can see that it is an irreducible variety.

The variety ${\cal M}_A$ has a natural symplectic form (on an open subvariety of smooth points). To see this, choose a basis in the representation vector space $V$ which is compatible with $q_i$'s. This choice is up to the action of the diagonal torus $T$. Projector $P$ is an element of the coadjoint orbit ${\cal O}_P$ of ${\bf GL}(n)$. Conditions $q_iPq_i=\lambda_iq_i$ determine the fiber over ${\bar \Lambda}$ of the moment map $\mu_A: {\cal O}_P\to \ttt^*$ for the action of $T$ (= fix diagonal coefficients of the matrix $P$ in the chosen basis), where $\ttt$ is the Cartan subalgebra. This means that the moduli space ${\cal M}_A$ is the symplectic reduction of the coadjoint orbit. Hence it is endowed with a symplectic form.

Let $\lambda_j=\sum_i \lambda_{ij}$.  Then there is a homomorphism $\psi: A_{n,{\bar \Lambda }}\to B_{\bf \Lambda}(\Gamma_{kn} )$ that takes $q_i\mapsto q_i$ and $P\mapsto \sum_{i=1}^kp_i$, which is in fact monomorphic.
It allows us to consider the pull-back morphism $\Psi: {\cal M}_B\to {\cal M}_A$.

\begin{Theorem}
The pull-back of the symplectic structure on ${\cal M}_A$ along $\Psi$ to ${\cal M}_B$ is zero.
\end{Theorem}
\begin{proof} According to proposition \ref{mb}, ${\cal M}_B$ is identified with the quotient of ${\rm pr}_1^{-1}F_{\lambda_{ij}}\cap X$ by the action of $T$. The map $\Psi :{\cal M}_B\to {\cal M}_A$ can be understood as the composite of the restriction to this intersection of the second projection ${\rm pr}_2:Y\times {\cal O}_{-P}\to {\cal O}_{-P}$, which takes it to a fibre $\mu^{-1}_A({\bar \Lambda})$ of the moment map $\mu_A$, with the map on the fiber, the factorization by the action of $T$.

Since the symplectic form on ${\cal M}_A$ is obtained by symplectic reduction, it is pulled back to the fibre of the $\mu_A$ to a restriction of the standard symplectic form $\omega_{-P}$ on the orbit of $-P$ (up to sign). Since $X$ is Lagrangian in $Y\times {\cal O}_{-P}$ the restriction of the form ${\rm pr}_2^*(\omega_{-P})$ to it coincides with the restriction of the form ${\rm pr}_1^*(\omega_Y)$, where $\omega_Y$ is the symplectic form on $Y$, the product of coadjoint orbits. The letter form is zero on ${\rm pr}_1^{-1}F_{\lambda_{ij}}\cap X$, because $F_{\lambda_{ij}}$ is Lagrangian in $Y$. Hence the result.
\end{proof}
In the terminology of noncommutative differential geometry, there is a noncommutative symplectic form on $A_{n,{\bar \Lambda }}$ which maps to the trivial noncommutative 2-form on $B_{\bf \Lambda}(\Gamma_{kn} )$. The noncommutative symplectic form induces ordinary symplectic forms on the moduli spaces of representations, hence the result of the theorem.
The theorem implies that the image of the map $\Psi : {\cal M}_B\to {\cal M}_A$ is isotropic with respect to the symplectic form.

\begin{Conjecture}
The image of $\Psi$ is Lagrangian in ${\cal M}_A$ for every irreducible component of ${\cal M}_B$. If $k\le \frac n2$, then the map is quasifinite on its image.
\end{Conjecture}
It is not hard to check that the dimension of the irreducible components of ${\cal M}_B$ has the half of the dimension of ${\cal M}_A$ as its lower bound.
The conjecture implies that this bound is saturated if $k\le \frac n2$. An example when it is not so for $k>\frac n2$ is when $n=6$. If $k=4$, then ${\cal M}_B$ has two irreducible components, one of dimension $3$ and the other one of dimension $4$, while ${\cal M}_A$ is of dimension $6$. If $k=5$, or $6$,
there is a component ${\cal M}_B$ of dimension $4$, this is the main result of \cite{BZh1}.

\subsection{A pair of algebraically unbiased systems and a symplectic embedding of the cotangent bundle}
The problem of classifications of two algebraically unbiased complete systems of projectors is about constructing the moduli of $n$-dimensional representations for algebra $B_{\bf \Lambda}(\Gamma_{nn} )$ for the full bipartite graph $\Gamma_{nn}$ with two rows of the same length $n$.

For the case of the graph $\Gamma_{nn}$, we have a more nice interpretation of the moduli space as the intersection of two Lagrangian subvarieties in a symplectic variety. We will identify an open subset in the symplectic variety with the total space of the cotangent bundle to one of this Lagrangian subvariety. This leads the identification of the moduli space with critical points of a suitable potential function. The function is similar to Landau-Ginzburg potential for some Fano varieties.

In case $k=n$, the sum of $n$ orthogonal projectors is the identity matrix. This means that the orbit ${\cal O}_{-P}$ consists of one point. Thus,
$$
Y={\cal O}_{p_1}\times \dots \times {\cal O}_{p_n},
$$
and $X$ is a Lagrangian subvariety in $Y$, the pre-image of the identity matrix under moment map $\mu :Y\to \g^*$.

Since the (formal) neighborhood of a Lagrangian subvariety $L$ in a symplectic variety is symplectically isomorphic to the (formal) neighborhood of the Lagrangian variety embedded as the zero section into the cotangent bundle to $L$, it is tempting to identify a Zariski open subvariety in $Y$ with the cotangent bundle to $X$. We construct an open symplectic embedding of the total space of the cotangent bundle to $L$ into $Y$:
$$
\Omega^1X\to Y.
$$
To understand this embedding, consider a projection from a Zariski open subset in $Y$ to $X$. Fibers of this projection will correspond to fibers of the vector bundle $\Omega^1X$ embedded in $Y$. A point in $Y$ is a set of $n$ rank 1 projectors. Let $U\subset Y$ be the open subset of those $n$-tuples of projectors whose images are linearly independent 1-dimensional vector spaces. To any such a set of projectors $\{r_i\}$, one can assign the unique complete system of {\em orthogonal} projectors $\{p_i\}$ with the same images:
$$
{\rm Im}p_i={\rm Im}r_i.
$$
The fiber of the projection could be parameterized a follows. A projector with a given image, a line $l$, is uniquely defined by its kernel $H$, a hyperplane that does not contain $l$. Hence, one can parameterize projectors with the image ${\rm Im}p_i$ by points $h\in {\mathbb P}V^*$ which are in the complement to a hyperplane $<h,l>=0$. This is an affine space. The fibre over the point $\{p_i\}$ is the product of these affine spaces over all $i$. It has a distinguished point, $\{p_i\}$ itself, hence it has a structure of the vector space.

In calculations, we will use the fact that $X$ is a homogenous $G$-space for $G={\bf GL}(n)$. Fix a reference basis in $V$ and consider orthogonal projectors $\{ q_i\}$ compatible with this basis. Let $g\in G$ be the transition matrix from the original basis to a basis compatible with a varying system of orthogonal projectors $\{ p_i\}$:
$$
p_i=gq_ig^{-1}.
$$
Elements $g$ in the diagonal torus $T\subset G$ preserve the basis $\{ q_i\}$. This gives identification $X=G/ T$.

Further, we identify the cotangent bundle to $X$ as a symplectic reduction of the cotangent bundle to $G$ with respect to the action of $T$.
We identify elements of the tangent spaces to $G$ (at every point $g$) with matrices in the basis compatible with $\{ q_i\}$. Then elements of the contangent space can also be considered as matrices via Killing pairing ${\rm Tr}AB$. Thus, we consider elements of $\Omega^1G$ to be pairs $(g,A)\in G\times {\rm Mat}_{n\times n}$. The symplectic form on the cotangent bundle reads:
\begin{equation}\label{omegax}
\omega_X={\rm Tr}{\rm d}A{\rm d}g
\end{equation}
where ${\rm d}A$ and ${\rm d}g$ are understood as matrix valued 1-forms. The product is taken henceforth in the space of matrix valued skew-symmetric forms.

Action of the diagonal element $t\in T$ on the cotangent bundle is given on $(g, A)\in \Omega^1G= G\times {\rm Mat}_{n\times n}$ by
$$
(g, A)\mapsto (gt^{-1}, tA).
$$
On can easily calculate the moment map $\mu_T:\Omega^1G\to \ttt^*$ with respect to the action of $T$:
$$
\mu_T(g, A)={\rm diag}(Ag),
$$
where $\ttt^*$ is identified with diagonal matrices and diag means the diagonal of the matrix.
Via symplectic reduction, we have an identification $\Omega^1X$ with $\mu^{-1}(0)/T$.

The map $\phi : \Omega^1X\to Y$ is defined on an element $(g,A)\in \mu^{-1}(0)$ by the formula:
\begin{equation}\label{rga}
r_i=p_i(1+gA)=gq_ig^{-1}(1+gA).
\end{equation}
Note that $r_i$ is invariant with respect to the action of $T$. Hence the map is well-defined. Moreover, since the diagonal of $Ag$ is zero, it follows that:
$$
r_i^2=gq_ig^{-1}(1+gA)gq_ig^{-1}(1+gA)=gq_i(1+Ag)q_ig^{-1}(1+gA)=gq_ig^{-1}(1+gA)=r_i,
$$
i.e. $r_i$'s are projectors of rank 1.

\begin{Theorem}
Map $\phi :\Omega^1X\to Y$ is an open symplectic embedding.
\end{Theorem}
\begin{proof}
The pull-back along $\phi$ of the symplectic form on $Y$ reads as:
$$
{\rm Tr}\sum_i r_i{\rm d}r_i{\rm d}r_i.
$$
It is a tedious but straightforward calculation (which we skip) to see that if we substitute $r_i$'s by formula (\ref{rga}), then the result coincides with the symplectic form (\ref {omegax}).
\end{proof}
{\bf Remark.} It would be interesting to have a more conceptual explanation for the coincidence of symplectic forms on $\Omega^1X$.

\section{Critical points of the potential and mirror symmetry}
Here we show that the solutions of our problem are critical points of a potential function, which can be chosen to be a Laurent polynomial. By mirror symmetry it corresponds a toric variety which we study here too.
\subsection{The potential}

Since fibers of the integrable system on $Y$ are Lagrangian and the map $\phi$ identifies the cotangent bundle to $X$ with the Zariski open set $U$ in $Y$, then the pre-images of the fibers of the integrable system along $\phi$ are Lagrangian subvarieties in the cotangent bundle. It is well-known that locally around points where a Lagrangian subvariety projects biholomorphically on $X$, it is identified with the section of the cotangent bundle, the differential of a potential function. Such a function might be multi-valued, if we maximally holomorphically extend it to $X$. Let us calculate this potential function.

The equation of the fiber $F_{\lambda_{ij}}$ is ${\rm Tr}(r_iq_j)=\lambda_{ij}$, or, in view of (\ref{rga}):
\begin{equation}\label{qalam}
{\rm Tr}(p_i(1+gA)q_j)={\rm Tr} (gq_ig^{-1}(1+gA)q_j)=\lambda_{ij}.
\end{equation}

Since $X=G/T$, a (local) holomorphic function $F$ on $X$ defines a (local) function on $G$:
$$
F=F(g)=F(\{g_{ij}\}).
$$
The chain rule ${\rm d}F=\sum_{ij}\frac{\partial F}{\partial g_{ij}}{\rm d}g_{ij}$ implies that
$$
{\rm d}F={\rm Tr}(A{\rm d}g)
$$
where $A$ is a matrix with coefficients
$$
a_{ij}=\frac{\partial F}{\partial g_{ji}}
$$

Denote by $\{g_{ij}\}$ the coefficients of the matrix $g$ and by $\{{\hat g}_{ij}\}$ those of the inverse matrix $g^{-1}$.
If $F$ is the potential of the fiber $F_{\lambda_{ij}}$ of the integrable system, then by plugging in the above formula for $A$ in the equations (\ref{qalam}), we obtain equations for $F$:
$$
g_{ji}{\hat g}_{ij}+g_{ji}\frac{\partial F}{\partial g_{ji}}=\lambda_{ij}.
$$
which imply:
\begin{equation}\label{fderivative}
\frac{\partial F}{\partial g_{ji}}= \frac{\lambda_{ij}}{g_{ji}}-{\hat g}_{ij}.
\end{equation}

Note that
$$
{\hat g}_{ij}=\frac{\partial ({\rm log}{\rm det}g)}{\partial g_{ji}}
$$
This implies that up to a constant
$$
F= \sum_{ij} \lambda_{ij}{\rm log}g_{ji}- {\rm log}{\rm det}g
$$
If we denote by $\Lambda$ the matrix with coefficients $\lambda_{ij}$, then we can write the above equation in the matrix form:
$$
F={\rm Tr}(\Lambda {\rm log}g)- {\rm log}{\rm det}g,
$$
which implies that the exponent of $F$ has the form:
$$
{\rm e}^F=\frac{\prod (g_{ij})^{\lambda_{ji}}}{{\rm det}g}
$$
According to what was explained before, we have
\begin{Proposition}
The critical points of this function are solutions to the problem of classification of two complete systems of orthogonal projectors $\{ p_i\}$ and $\{ q_j\}$ satisfying condition (\ref{trpq}).
\end{Proposition}

Note that (\ref{fderivative}) implies that critical points of $F$ are solutions of the equations:
$$
\frac{\lambda_{ij}}{g_{ji}}={\hat g}_{ij}
$$

As ${\hat g}_{ij}$ are coefficients of the inverse matrix, the equation on coefficients of $g$ are:
\begin{equation}
\sum^{n}_{j=1}\frac{\lambda_{jk}g_{ij}}{g_{kj}} = 0.
\end{equation}
for all $i \ne k$.

\subsection{A Laurent polynomial and Landau-Ginzburg model potentials}
Denote by $G$ the inverse of the exponent for the potential $F$:
$$
E={\rm e}^{-F}=\frac{{\rm det}g}{\prod (g_{ij})^{\lambda_{ji}}}
$$
Note that critical points of this function are the same as those of $F$.

Function $E$ is a multi-valued function on the $n^2$ dimensional torus $({\mathbb C}^*)^{n^2}$ with coordinates $g_{ij}$. It is invariant with respect to the action of the diagonal torus $T$ on both sides of the matrix $g$. Thus, we can consider $E$ as a function on the torus $T\setminus ({\mathbb C}^*)^{n^2}/T$ of dimension $(n-1)^2$.

Now we will concentrate on the case when $\lambda_{ij}=\frac 1n$, for all $(ij)$. This corresponds to the case of algebraically mutually unbiased bases. Function $E$ reads:
\begin{equation}\label{e1n}
E=\frac{{\rm det}g}{\prod (g_{ij})^{\frac 1n}}
\end{equation}
Consider the isogenic $n$-cover ${\mathbb T}\to T\setminus ({\mathbb C}^*)^{n^2}/T$ by another torus ${\mathbb T}$ of dimension $(n-1)^2$ defined by adding to the holomorphic functions the $n$-th root of $\prod g_{ij}$. Thus we have a new multiplicative coordinate on this torus:
$$
z=\prod (g_{ij})^{\frac 1n}
$$
The advantage of the new torus is that the lift of $E$ to it (which we denote by the same letter $E$) is a Laurent polynomial in multiplicative coordinates on the torus. The right and left action of torus $T$ lifts to the $n$-fold cover and $E$ remains invariant with respect to it. Thus we can again diminish the dimension of the torus to $(n-1)^2$.

{\bf Example.}
Let $n=2$. After killing the action of $T\times T$, the matrix $g$ has the form:
\begin{equation}
g = \left(
\begin{array}{cc}
1 & 1\\
1 & x
\end{array}
\right)
\end{equation}

Hence $z^2=x$, and $E=z-\frac 1z$. After changing the coordinate $z$ for $iz$ one get that $E=i(z+\frac 1z)$. In other words, $E$ is (up to a multiplier) the Landau-Ginzburg potential corresponding to the projective line ${\mathbb P}^1$, the target space for the super-symmetric sigma-model.

The potential has finite number of critical points for $n=2,3, 5$. There is a curve of critical points for $n=4$, and a 4-dimensional family for $n=6$ (see \cite{BZh1}). In view of Winnie-the-Pooh Conjecture, one can assume that the structure of critical points depend on the prime number decomposition for $n$. There is a counter-example to the conjecture that the number of critical points is always finite for prime $n$: there is a 1-parameter family of critical points for $n=7$.

\subsection{The toric variety of Birkhoff polytope}

According to the prescription of mirror symmetry, we consider the toric variety $X_E$ whose defining fan has cones spanned by faces of the Newton polytope of the function $E$ on the torus ${\mathbb T}$ introduced above.

The Newton polytope is in fact the classically known Birkhoff polytope, the convex hull of permutation matrices in the space of real square matrices. We denote the {\em real} coordinates in the space of matrices by $\lambda_{ij}$ (this notation might look strange, but it has the reason that this coordinates are indeed dual to $g_{ij}$). Then Birkhoff polytope lies in the affine subspace $L$ defined by equations (\ref{sumi}) and (\ref{sumj}), which reflects the fact that the torus ${\mathbb T}$ has dimension $(n-1)^2$. Let ${\tilde N}$ be the lattice generated by all integer matrices and the matrix ${\Lambda}_o$ with all coefficients equal to $\frac 1n$. Define an integer lattice  $N$ in $L$ to be $L\cap {\tilde N}$. This is the lattice of characters of our torus ${\mathbb T}$ if we take the origin to be matrix ${\Lambda}_o$. The shift of the origin is justified by the denominator in (\ref{e1n}).

According to Birkhoff-von Neumann theorem, Birkhoff polytope is the intersection of $L$ with the set of matrices with non-negative coefficients. In other words, it is the locus of double-stochastic matrices. It was intensively studied and has many applications. There are $n^2$ facets of the polytope, given by equation $\lambda_{ij}=0$, for a particular choice of $(ij)$.

\begin{Proposition}
Variety $X_E$ is a Gorenstein toric Fano variety with terminal singularities.
\end{Proposition}
\begin{proof}
Function $l_{ij}=n\lambda_{ij}-1$ is zero at the origin, integer-valued on ${\tilde N}$ and has value $-1$ on the facet $\lambda_{ij}=0$.
This means that it defines an integer vertex of the dual polytope. Hence, all vertices of the dual polytope are integer, i.e. Birkhoff polytope is reflexive.
In terms of toric geometry, this means that our toric variety $X_E$ is Fano and Gorenstein.

Clearly, the polytope lies in the unit hypercube $[0,1]^{n^2}$. The only points of the lattice ${\tilde N}$ in the unit cube are vertices of the cube and integer multiples of the matrix ${\Lambda}_o$ with multiplier up to $n$. As Birkhoff polytope lies in $L$, this implies that the only integer points in Birkhoff polytope are ${\Lambda}_o$ and the vertices of the polytope. This means that the toric variety has terminal singularities.
\end{proof}

The toric variety $X_E$ is ${\mathbb P}^1$, for $n=2$, it is ${\mathbb P}^2\times {\mathbb P}^2$, for $n=3$, and it is singular for higher $n$.
The geometry of this toric variety as well as critical points of the potential deserve further study.

\def\cprime{$'$}
\ifx\undefined\bysame
\newcommand{\bysame}{\leavevmode\hbox to3em{\hrulefill}\,}
\fi


\begin{thebibliography}{10}

\bibitem[BBD]{BBD}
A.~Beilinson, J.~Bernstein, and P.~Deligne, {\em Faisceaux pervers},
Ast{\'e}risque, vol. 100, Soc. Math. France, 1983.

\bibitem[BZh]{BZh} A.~Bondal, I. Zhdanovskiy, {\em Representation theory for system of projectors and discrete Laplace operators}, IPMU13-0001, IPMU, Kashiwa, Japan, 2013, 48 pp.

\bibitem[BZh1]{BZh1} A.~Bondal, I. Zhdanovskiy, {\em Orthogonal pairs for Lie algebra sl(6)}, IPMU14-0296, IPMU, Kashiwa, Japan, 2014, 89 pp.


\bibitem[BTSW]{BTSW}
P. Oscar Boykin, Meera Sitharam, Pham Huu Tiep, Pawel Wocjan,
{\em Mutually Unbiased Bases and Orthogonal Decompositions of Lie Algebras}
http://arxiv.org/abs/quant-ph/0506089




\bibitem[G]{G} C.F.Gauss, Pentagramma Mirificum, Werke, Bd. III, 481 - 490; Bd
VIII,
106 - 111.


\bibitem[Kr]{Kr} Kraft, H., {\em Geometric methods in representation theory.} In: Representations
of Algebras, Workshop Proceedings (Puebla, Mexico 1980)., Lecture Notes
in Math. vol. 944, Springer Verlag, BerlinЦHeidelbergЦNew York, 1982

\bibitem[KT]{KT}
A.I. Kostrikin, Pham Huu Tiep, {\em Orthogonal decompositions of Lie
algebras and integral lattices}. Berlin: Walter de Gruyter, 1994.

\bibitem[N]{N} John Napier, {\em Mirifici Logarithmorum canonis descriptio}, Lugdini, 1619.

\bibitem[Sch]{Sch} V. Schectman, {\em Pentagramma mirificum and elliptic functions}, Annales de la Faculte des Sciences de Toulouse. Serie VI. Mathematiques 01/2013; 2(2), arXiv:1106.3633

\end{thebibliography}
\end{document}